\documentclass[11pt,a4paper]{amsart}

\usepackage{amscd,amssymb,amsopn,amsmath,amsthm,graphics,amsfonts,mathrsfs,accents,enumerate,verbatim,calc}
\usepackage{amssymb}
\usepackage{enumitem}
\usepackage[dvips]{graphicx}
\usepackage[colorlinks=true,linkcolor=red,citecolor=blue]{hyperref}
\usepackage{latexsym}
\usepackage{mathrsfs}
\usepackage{newtxmath}
\usepackage{newclude}
\usepackage{syntonly}
\usepackage{tikz}		
\usepackage{tikz-cd}	
\usepackage[all]{xy}
\usepackage{nicematrix}
\usepackage{arydshln} 	
\usepackage{silence}
\date{}
\hypersetup
{colorlinks=true,linkcolor=black}
\allowdisplaybreaks[4] \footskip=15pt
\renewcommand{\uppercasenonmath}[1]{}

\numberwithin{equation}{section} \theoremstyle{plain}
\newtheorem{lemm}{Lemma}[section]
\newtheorem{coro}[lemm]{Corollary}
\newtheorem{prop}[lemm]{Proposition}
\newtheorem{theo}[lemm]{Theorem}

\newtheorem{defn}[lemm]{Definition}
\newtheorem{exam}[lemm]{Example}

\definecolor{morpink}{RGB}{239,229,248}
\definecolor{morblue}{RGB}{146,172,209}
\definecolor{defcolor}{RGB}{50,205,50}
\definecolor{tcolor}{RGB}{210,105,30} 
\definecolor{lcolor}{RGB}{210,105,30} 
\definecolor{pcolor}{RGB}{253,249,238} 

\def\no{\noindent}

\def\mf{R}
\def\mr{\mathbb{R}}
\def\ca{A}
\def\cb{B}
\def\cc{C}

\def\lr{R}
\def\cs{B}

\def\ra{\rightarrow}

\def\proof{\no\rmm{Proof:}}
\def\endd{$\hfill\Box$}  				

\newcommand{\rmm}[1]{
	{\rm #1}
}

\newcommand{\tela}[1]{
	\tilde{#1}
}

\newcommand{\Rmnum}[1]{\uppercase\expandafter{\romannumeral#1}}

\newlist{rlist}{enumerate}{1}
\setlist[rlist]{label=\roman*\rmm{)}}

\begin{document}
\begin{center}
{\Large  \bf  Frobenius extensions about centralizer matrix algebras}

\vspace{0.5cm} Qikai Wang and Haiyan Zhu\footnote{Corresponding author

\ Supported by the National Natural Science Foundation of China  (12271481).}  \\
School of Mathematical Science, Zhejiang University of Technology, Hangzhou 310023, China

E-mails: qkwang@zjut.edu.cn and hyzhu@zjut.edu.cn
\end{center}

\bigskip
\centerline { \bf  Abstract}
\leftskip10true mm \rightskip10true mm \noindent

This paper investigates the conditions under which the centralizer algebra $S_n(c,R)$  of a matrix \( c\in M_n(R)\) is a (separable) Frobenius extension of the base algebra $R$. 
For an algebra $R$ over an integral domain $\mathbb{k}$, we provide necessary and sufficient conditions for $S_n(c,R)/R$ to be a (separable) Frobenius extension when \(c\) is in Jordan canonical form with eigenvalues in \(\mathbb{k}\). 
We extend this analysis to arbitrary matrices over a field and derive conditions for matrix diagonalizability through Frobenius extensions.
\\[2mm]

\leftskip0true mm \rightskip0true mm


\section {Introduction}

Let $\lr$ be a unitary ring and $n$ a natural number.
We set $[n]:=\{1,2,\dots, n\}$ and denote by $M_n(\lr)$ the ring of $n\!\times\! n$ matrices over $\lr$ and by $e_{i,j}$ the matrix units of $M_n(\lr)$. 

For a nonempty set $\cc$ of matrices in $M_n(\lr)$, we define the \emph{centralizer algebra} of $\cc$ by
\begin{center}
	$S_n(\cc,\lr):=\{a\in M_n(\lr)\mid  ac=ca$ for all $c\in\cc \}$.
\end{center}
In case $C=\{c\}$, we write $S_n(c,\lr)$ for $S_n(\{c\},\lr)$.

The study of centralizer matrix algebras  is one of important topics in the theory of matrices, which has significant applications not only in abstract studies, but also in applied problems  (see \cite{MR201472,MR85214,MR168568,xizhang2022}).
If $\cc$ consists of nilpotent matrices and $\lr$ is an algebraically closed field, then the variety consisting of nilpotent matrices in $S_n(\cc,\lr)$ is of great interest in understanding properties of semisimple Lie algebras (see \cite{MR2363491,MR2018787}). 
Typical examples of centralizer matrix algebras include centrosymmetric matrix algebras (see \cite{MR820054,xi1}), the Auslander algebras of the truncated polynomial algebras (see \cite{MR4241258}) and centralizer of a matrix  over an algebraically closed field. 
It is commonly known that centrosymmetric matrices have significant applications in Markov processes (see \cite{MR820054}), engineering problems and quantum physics (see \cite{MR998028}).

An algebra extension $\ca/\cs$ is called a Frobenius extension provided that $\ca$ is finitely generated projective as a right $\cs$-module and isomorphic to Hom$(\ca_\cs,\cs_\cs)$ as a $\cs$-$\ca$-bimodule.
They play a crucial role in various mathematical domains.
These include representation theory (see \cite{Xi-fbfd,xi1}), knot theory, and solutions to the Yang-Baxter equation (see \cite{MR1401512}). They also have significant applications in topological quantum field theories and the theory of codes (see \cite{Gnilke-etbwfrfb,Kock-fatqft}).

An interesting example of Frobenius extensions  related to centralizer matrix algebras, due to Xi and Zhang (see \cite{MR4241258}), is that the extension $M_n(\lr)/ S_n(c,\lr)$ is always a separable Frobenius extension for an arbitrary field $R$ and $c\in M_n(\lr)$.
Another interesting example of Frobenius extensions about matrix algebras is that $M_n(\lr)/\lr$ is always a Frobenius extension for any algebra $R$ (see \cite{Kanzaki-ngffe}).
However, it remains unknown if $S_n(c,\lr)/\lr$ is a Frobenius extension.
Here, the general question reads as follows.

\textit{Is $S_n(c,\lr)/\lr$ always a Frobenius extension?}

\begin{center}
	\begin{tikzcd}
		M_n(\lr) &  & {S_n(c,\lr)} \arrow[ll ] &  & \lr \arrow[ll, "?" description] \arrow[llll, bend right]
	\end{tikzcd}
\end{center}

The main purpose of this note is to provide an answer  to the above question. To state our results precisely, we first introduce some notation and definitions.

Recall that  a Jordan block, denoted by $J_n(\lambda) $, is an $n \times n$ matrix with its diagonal filled with the eigenvalue $\lambda$ and ones on the superdiagonal, and zeros elsewhere.
Let $J_n$ denote the $n \times n$ Jordan block with eigenvalue zero.
A matrix $c$ is \emph{Jordan-similar} to $a$ if $c = u^{-1}au$ for some invertible matrix $u$, where $a$ is in Jordan canonical form. This form is characterized by having Jordan blocks along the diagonal, with each block corresponding to an eigenvalue of $c$.

Since every unitary ring $R$ is an algebra over the ring of integers $\mathbb{Z}$, we adopt the term `algebra' instead of `ring' throughout this paper.
Our  main result reads as follows.

\begin{theo}\label{main2}
	Let $\lr$ be an algebra over an integral domain $\vmathbb{k}$ and $c=\bigoplus J_{n_i}(\lambda_i)$  a Jordan-block matrix in $M_n(R)$
    with eigenvalues $\lambda_i \in   \vmathbb{k}  $.
	Then
	\begin{rlist}
		\item $S_n(c,\lr)/\lr$ is a Frobenius extension if and only if  $n_i = n_j$ whenever $\lambda_i = \lambda_j$ for any pair $i$ and $j$.
		\item
		If every element of $[n]$ is a unit in $R$, then  $S_n(c,\lr)/\lr$ is a separable Frobenius extension if and only if  $n_i = 1$ for any $i$.
	\end{rlist}
\end{theo}

For a field $R$, an $R$-algebra $A$ is a (separable) Frobenius $R$-algebra if $A/R$ is a (separable) Frobenius extension.
Additionally, we denote by  $\text{char}(R)$ the characteristic of the field $\mf$ and   $\overline{\mf}$ the algebraic closure of $\mf$.
Inspired by the method of Xi and Zhang in \cite{MR4461655}, 
we obtain the following corollary.

\begin{coro}\label{coro0}
	Let $\lr$ be a field and	$c$   a  matrix   in $M_n(\lr)$.
	Assume that $c$ is similar to $\bigoplus_i J_{n_i}(\lambda_i)$ in $M_n(\overline{R})$ with $\lambda_i\in\overline{R}$.
	Then
    \begin{rlist}
        \item $S_n(c,\lr)$ is a Frobenius $R$-algebra if and only if  $n_i = n_j$ for $\lambda_i = \lambda_j$.
        \item $S_n(c,\lr)$ is a separable Frobenius  $R$-algebra if and only if  $n_i = 1$ for any $i$.
    \end{rlist}
\end{coro}

Our research provides a way to determine whether the extension $S_n(c,R)/R$ is Frobenius or not.
This allows us to present numerous simple examples of non-Frobenius extensions.
Additionally, Corollary \ref{coro0} naturally gives rise to following corollary.

\begin{coro}\label{coro1}
	Let $\lr$ be an algebraically closed field.
    Then a matrix $c\in M_n(\lr)$ is diagonalizable  if and only if $S_n(c,\lr)$ is a separable Frobenius  algebra.
\end{coro}

Furthermore, minimal matrices are also an important concept in centralizer algebras.
A  nonscalar matrix $c$ is  \emph{minimal}, if for every $a\in M_n(\mf)$ with $S_n(a,\mf)\subseteq S_n(c,\mf)$ it follows that $S_n(a,\mf)=S_n(c,\mf)$.
By Corollary \ref{coro0} and \cite[Proposition 2.3]{MR3018047}, we have the following corollary.
\begin{coro}\label{coro4}
	Let  $\mf$ be a  field with   $\text{char}(R) \geqslant n$ $(n\geqslant 2)$.
	If  $c\in M_n(R)$ is minimal, then $S_n(c,R)$ is a Frobenius algebra.
\end{coro}

This article is outlined as follows. Section \ref{pre} is dedicated to some foundational concepts of  Frobenius extensions. Section \ref{mainsec} presents the proof of Theorem \ref{main2}. In addition,  Section \ref{mainsec} also gives corollaries and examples which contribute to showing the main result.

\section{Preliminaries}\label{pre}

In this section we discuss basic properties of separable Frobenius extensions and centralizer matrix algebras.
Throughout the paper, $\lr$ denotes an algebra over an integral domain $\vmathbb{k}$.
In addition, we denote by  $M_{m\times n}(\lr)$ the set of all $m \times n$ matrices over $\lr$ and by $e_{i,j}$ the matrix units of $M_{m\times n}(\lr)$ with $i \in [m], j \in [n]$. We write $I_n$ for the identity matrix in $M_n(\lr)$. For a matrix $c \in M_{m\times n}(\lr)$, we denote by   $tr(c)$ the trace of $c$.

First, let us recall some facts about separable Frobenius extensions.

\begin{lemm}\cite{newexam}
	The following are equivalent:
    \begin{rlist}
        \item  $\ca/\cs$ is a Frobenius extension
        \item There exists $E\in \rmm{Hom}_{\cs-\cs}(\ca,\cs)$, $X_i, Y_i\in \ca$ such that $\forall a\in \ca$
        \begin{center}
            $\sum_{i} X_i E(Y_i a) =a$, and $\sum_{i} E(aX_i)Y_i =a$.
        \end{center}
    \end{rlist}
\end{lemm}

$(E,X_i,Y_i)$ is called a Frobenius system.

\begin{lemm}\cite{newexam}
	Suppose that $\ca/\cs$ is a Frobenius extension with system $(E,X_i,Y_i)$.
	Then $\ca$ is a separable extension of $\cs$ if and only if there exists $d\in C_\ca(\cs):=\{a\in\ca \mid  ab=ba$ for all $b\in\cs \}$ such that $\sum_i X_i d Y_i= 1$.

\end{lemm}

Next we show the transitivity of (separable) Frobenius extensions.
\begin{prop}\label{composite}\cite{newexam}
	Let $\cb/\ca$ be a (separable) Frobenius extension, and let $\cc/\cb$ be a (separable) Frobenius extension. Then $\cc/\ca$ is a (separable) Frobenius extension.
\end{prop}

\begin{lemm}\label{chai}
	 Let $B^i$ $(i\in [n])$ and $A$ be algebras.
     Then $\cb^i/\ca$ is a (separable) Frobenius extension  for any $i\in [n]$ if and only if  $\bigoplus_{i\in [n]}\cb^i/\ca$ is a (separable) Frobenius extension.
\end{lemm}

\proof
We only need to prove the case of $n=2$, and it  follows easily by induction on $n$.

We first show that if $\cb^1/\ca$ and $\cb^2/\ca$ are Frobenius extensions, then $\cb^1 \bigoplus \cb^2/A$ is a Frobenius extension.
Assume that $(E^1,X^1_i,Y^1_i)$ and $(E^2,X^2_i,Y^2_i)$ are Frobenius systems of $\cb^1/\ca$ and $\cb^2/\ca$.
For any $(b^1,b^2)\in\cb^1\bigoplus \cb^2$, let
$E(b^1,b^2) = E(b^1,0)+E(0,  b^2)=E^1(b^1)+ E^2(b^2)$.
It is clear that  $E\in \rmm{Hom}_{\ca-\ca}(\cb^1\bigoplus \cb^2,\ca)$.
For any $(b^1,b^2)\in \cb^1\bigoplus \cb^2$, there is
\begin{align*}
	&\sum_i E [(b^1,b^2)(X^1_i,0)](Y^1_i,0)+\sum_i E [(b^1,b^2)(0,X^2_i)](0,Y^2_i)\\
	=&\sum_i E(b^1X^1_i,0)(Y^1_i,0)+\sum_i E(0,b^2X^2_i)(0,Y^2_i)\\
	=&\sum_i (E^1(b^1X^1_i),0)(Y^1_i,0)+\sum_i (0,E^2(b^2X^2_i))(0,Y^2_i)\\
	=&(\sum_iE^1(b^1X^1_i)Y^1_i,0)+(0,\sum_iE^2(b^2X^2_i)Y^2_i)\\
	=&(b^1,b^2).
\end{align*}
Similarly, we have $\sum_i (X^1_i,0)E [(Y^1_i,0)(b^1,b^2)] +\sum_i (0,X^2_i)E [(0,Y^2_i)(b^1,b^2)]=(b^1,b^2)$.

On the other hand, assume $\cb^1 \bigoplus \cb^2/A$ is a Frobenius extension.
Let $(E,(X^1_i,X^2_i),(Y^1_i,Y^2_i))$ be a Frobenius system  of $(\cb^1 \bigoplus \cb^2)/A$.
Let $E^1(b^1) = E(b^1, 0)$ and $E^2(b^2) = E(0, b^2)$, for $b^1\in \cb^1$ and $b^2\in\cb^2$.
Then it is easy to check $(E^1,X^1_i,Y^1_i)$ and $(E^2,X^2_i,Y^2_i)$ are Frobenius systems.

For the separable Frobenius extension, assume there exist $d^1\in \cb^1$ and $d^2\in \cb^2$ such that $\sum_i X_i^1 d^1 Y_i^1 =1$ and  $\sum_i X_i^2 d^2 Y_i^2 =1$.
Then $$\sum_i (X_i^1,0) (d^1,d^2) (Y_i^1,0)+\sum_i (0,X_i^2) (d^1,d^2) (0,Y_i^2) =  (\sum_iX_i^1d^1Y_i^1,\sum_iX_i^2d^2Y_i^2)=(1,1).$$

Similarly, there exists $d=(d^1,d^2)\in C_{\cb^1\bigoplus \cb^2}(\ca)$ such that $\sum_i (X^1_i,X^2_i)d (Y^1_i,Y^2_i)=(1,1)$. This implies that $\sum_i X_i^1 d^1 Y_i^1 =1$ and $\sum_i X_i^2 d^2 Y_i^2 =1$.
\endd

The direct sum of matrices is defined as follow.

\begin{defn}
	Given an $m \times n$ matrix $a$    and a $p \times q$ matrix   $b$, their direct sum, denoted by $a \bigoplus b$, is defined as a $(m+p) \times (n+q)$ matrix:
	$$ a \bigoplus b = \begin{bmatrix} a & 0 \\ 0 & b \end{bmatrix},$$
	where the zero matrices have appropriate sizes to fill the blocks, thereby making $a$ and $b$ the diagonal blocks of the resulting matrix.
\end{defn}

Secondly, let us recall some facts about centralizer matrix algebras.

\begin{lemm}\cite{MR4241258}
	Let  $c,a\in M_n(\lr)$ and $u\in GL_n(\lr)$ with $a = u^{-1}cu$. 
	Then, $S_n(a,\lr) = u^{-1}S_n(c,\lr)u:=\{ u^{-1}bu \mid  b\in S_n(c,\lr) \}$.
\end{lemm}

Recall that $n\times n$ matrix $a$ is called \emph{semicirculant} if it has the form
\begin{center}
	$\begin{bmatrix}
		a_1&a_2&a_3&\dots&a_n\\
		0&a_1&a_2&\dots&a_{n-1}\\
		0&0&a_1&\dots&a_{n-2}\\
		\vdots&\vdots&\vdots&\ddots &\vdots\\
		0&0&0&\dots&a_1\\
	\end{bmatrix}\in M_n(R)$.
\end{center}
It should be noted that a semicirculant matrix is a special case of $S_n(c,R)$, where $c$ is the Jordan block $J_n$.

\begin{lemm}\cite{1986}\label{block}
	Let  $a=\sum_{i=1}^{m}\sum_{j=1}^n a_{i,j}e_{i,j}$ be an $m\times n$ matrix over $\lr$.
	Suppose $J_{m}(\lambda_1)a = a J_{n}(\lambda_2)$ with $\lambda_1,\lambda_2\in \vmathbb{k}$.
	\begin{rlist}
	\item If $\lambda_1 = \lambda_2$,
	then $a_{i,j}=a_{i+1,j+1}$, $a_{i+1,1}=0$ and $a_{m,j-1}=0$ for any $1\leqslant i\leqslant m-1$, $2\leqslant j\leqslant n$.
	\item If $\lambda_1 \neq \lambda_2$, then $a=0$.
	\end{rlist}
\end{lemm}

\begin{exam}\label{exam1}
	By Lemma \ref{block}, we have:
	\begin{rlist}
	\item  $\{ a\in M_3(\lr)\ \mid \  J_3a=aJ_3 \}  =\left\{\left.\begin{bmatrix}
		a_0&a_1  &a_2 \\
		0&a_0  &a_1 \\
		0&0&a_0
	\end{bmatrix}  \right| \  a_0,a_1,a_2\in\lr\right\}$,
	\item $\{ b\in M_{4\times2}(\lr)\ \mid \  J_4b=bJ_2 \}  =\left\{\left.\begin{bmatrix}
	b_0&b_1   \\
	0&b_0    \\
	0&0\\
	0&0
	\end{bmatrix}  \right| \  b_0,b_1 \in\lr\right\}$,
	\item $\{ c\in M_{2\times4}(\lr)\ \mid \ \  J_2c=cJ_4 \}  =\left\{\left.\begin{bmatrix}
		0&0  &c_2&c_3 \\
		0&0&0 &c_2
	\end{bmatrix}     \right| \  c_2,c_3 \in\lr\right\}$.
	\end{rlist}
\end{exam}

For convenience, we define:
$$ J_{m,n}^k=\left\{
	\begin{aligned}
		\sum_{i=1}^{n-k} e_{i,i+k}, && \text{if\ } k < n,\\
		0,\quad&& \text{if\ } k\geqslant n.
	\end{aligned}  \right. $$
Here $k$ is a non-negative integer.

For example:
\begin{rlist}
\item $J_{2,4}^2=\begin{bmatrix}
	0&0  &1&0 \\
	0&0&0 &1
\end{bmatrix}$,
\item $J_{n,n}^1$ is simply the $n\times n$ Jordan matrix $J_n$, and
\item $J_n^i=J_{n,n}^i = (J_n)^i$.
\end{rlist}

Note that Xi and Zhang in \cite{MR4241258} also introduced a comparable structure, denoted as $G^p=J_{m,n}^{n-p}$, for $m \times n$ matrices. 
This paper employs $J_{m,n}^i$ over $G^p$ to describe semicirculant matrices.

So we can rewrite Lemma \ref{block} as follows.
 
\begin{coro}\label{block2}
	For any $\lambda_1,\lambda_2 \in R$, We have:
	$$ \{a\in M_{m\times n}(R)\mid J_m(\lambda_1)a = a J_n(\lambda_2)\} = \left\{\begin{aligned}
		\left\{\left. \sum_{i={\rm Max}\{0,n-m\}}^{n-1} a_i \cdot J_{m,n}^i\ \right|\  a_i\in\lr \right\},& & \text{if\ } \lambda_1 = \lambda_2,\\ 
		\{0\},\quad\quad\quad\quad\quad& &\text{if\ } \lambda_1 \neq \lambda_2.
	\end{aligned} \right.$$
\end{coro}

For convenience,  we provide some operational rules.

\begin{prop}\label{mul}
	Let $l,m,n,k_1, k_2$ be positive integers with $k_2 \geqslant \max\{0, n-m\}$.
	Then we have:
	$$J_{ l,m }^{k_1} \cdot  J_{m, n}^{k_2} = J_{ l,n }^{k_1+k_2}.$$
\end{prop}
\proof
Since $J_{l,m}^{k_1} =   \sum_{i=1}^{m-k_1} e_{i,i+k_1}$ and $J_{m,n}^{k_2} =   \sum_{i=1}^{n-k_2} e_{i,i+k_2}$, we have following equations:
\begin{align*}
	J_{l,m}^{k_1} \cdot J_{m,n}^{k_2}  &=   \sum_{i=1}^{m-k_1} e_{i,i+k_1}\cdot   \sum_{j=1}^{n-k_2} e_{j,j+k_2}\\
	&=   \sum_{i=1}^{m-k_1}\sum_{j=1}^{n-k_2} e_{i,i+k_1}\cdot e_{j,j+k_2}\\
	&=  \sum_{i=1 }^{min\{m-k_1, n-k_1-k_2\}}   e_{i,i+k_1+k_2}\\
	&=  \sum_{i=1 }^{n-k_1-k_2}   e_{i,i+k_1+k_2}\quad  (\text{since } k_2 \geqslant \max\{0, n-m\})\\
	& =J_{m,n}^{k_1+k_2}.
\end{align*} 
\endd
\section{The proof of Theorem \ref{main2}}\label{mainsec}

This section is devoted to proving the Theorem \ref{main2}. We begin with the following lemma concerning the structure $S_n(J_n,\lr)/\lr$.

\begin{lemm}\label{trivial}

For any positive integer $n$,
$S_n(J_n,\lr)/\lr$ is a Frobenius extension.
Moreover, $S_n(J_n,\lr)/\lr$ is  a separable Frobenius extension if and only if $n=1$.

\end{lemm}

\proof
According to Lemma \ref{block2}, we have:
$$S_n(J_n,\lr) = \{\sum_{i=0}^{n-1} \gamma_i\cdot J_n^i\mid  \gamma_i\in \lr   \}.$$
To show that the extension $S_n(J_n,\lr)/\lr$ is a Frobenius extension, we define:
\begin{center}
	$E: S_n(J_n,\lr) \ra  \lr$, $\sum_{i=0}^{n-1} \gamma_i\cdot J_n^i\mapsto \gamma_{n-1} $, $X_i = J_n^i$ and $Y_i = J_n^{n-1-i},$ 
\end{center}
for $i=0,1,\dots , n-1$.

In the following, we will prove that $(E,X_i,Y_i)$ is a Frobenius system.
The proof will be structured in two steps.

Step1: $E$ is a morphism of $\lr\!-\!\lr \!-\!$bimodules.

In fact, for any $\sum_{i=0}^{n-1} \gamma_i\cdot J_n^i\in S_n(J_n,\lr)$, $\sum_{i=0}^{n-1} \gamma_i'\cdot J_n^i\in S_n(J_n,\lr)$, we get:
$$E\left(\sum_{i=0}^{n-1} \gamma_i\cdot J_n^i+\sum_{i=0}^{n-1} \gamma_i'\cdot J_n^i\right)= E\left(\sum_{i=0}^{n-1} (\gamma_i+\gamma_i')\cdot J_n^i\right)=  \gamma_{n-1} +\gamma_{n-1}'  =E\left(\sum_{i=0}^{n-1} \gamma_i\cdot J_n^i\right)+E\left(\sum_{i=0}^{n-1} \gamma_i'\cdot J_n^i\right).$$
Moreover, for any $\gamma\in \lr$ and $\sum_{i=0}^{n-1} \gamma_i\cdot J_n^i\in S_n(J_n,\lr)$, it is easy to check that:
$$E\left(\gamma\cdot \sum_{i=0}^{n-1} \gamma_i\cdot J_n^i\right) = E\left(\sum_{i=0}^{n-1} r\cdot \gamma_i\cdot J_n^i\right)= \gamma\cdot \gamma_{n-1} = \gamma\cdot E\left(\sum_{i=0}^{n-1} \gamma_i\cdot J_n^i\right),$$
and $$E\left(\sum_{i=0}^{n-1} \gamma_i\cdot J_n^i\cdot \gamma \right) = E\left(\sum_{i=0}^{n-1} \gamma_i\cdot \gamma\cdot J_n^i\right)= \gamma_{n-1}\cdot \gamma = E\left(\sum_{i=0}^{n-1} \gamma_i\cdot J_n^i\right)\cdot \gamma.$$

Step2:   One has  $\sum_{i} X_i E(Y_ia) =a = \sum_{i} E(aX_i)Y_i$, for any $a = \sum_{j=0}^{n-1} \gamma_j\cdot J_n^j \in S_n(J_n,\lr)$.

In fact,
\begin{align*}
	\sum_{i=0}^{n-1} E(aX_i)Y_i &=\sum_{i=0}^{n-1} E\left(\sum_{j=0}^{n-1} \gamma_j\cdot J_n^jJ_n^i\right)J_n^{n-1-i}\\
	&=\sum_{i=0}^{n-1} E\left(\sum_{j=0}^{n-1} \gamma_j\cdot J_n^{j+i}\right)J_n^{n-1-i}\\
	&= \sum_{i=0}^{n-1}  \gamma_{n-1-i}\cdot J_n^{n-1-i}\\
	&=a.
\end{align*}
Similarly, $\sum_{i} X_i E(Y_ia) =a$.

Thus  $(E,X_i,Y_i)$ is a Frobenius system and $S_n(J_n,\lr)/\lr$ is a Frobenius extension.

In addition, for any $\gamma\in\lr$, 
$$\sum_{i=0}^{n-1} X_i\cdot \gamma\cdot Y_i= \sum_{i=0}^{n-1} J_n^i\cdot \gamma\cdot J_n^{n-1-i} = n\cdot \gamma\cdot J_n^{n-1}.$$ 
If $n>1$, $\sum_{i=0}^{n-1} X_i \cdot \gamma\cdot  Y_i\neq I$ (where I is the identity element in $S_n(J_n,\lr)$), that is, $S_n(J_n,\lr)/\lr$ is not  a separable Frobenius extension.
\endd
\vspace{10pt}

The following conclusions are well-known, and the first part can be found in \cite{Kanzaki-ngffe}. Here, we provide a short proof for completeness.


\begin{lemm}\label{RI,Mn}
    The extension $M_n(\lr)/\lr$ is a  Frobenius extension.
	Furthermore,  $M_n(\lr)/\lr$ is a separable Frobenius extension if and only if
    $n$ is invertible in $R$.
\end{lemm}
\proof
For any $a=\sum_{i=1}^{n}\sum_{j=1}^{n} a_{(i,j)} e_{i,j} \in M_n(\lr)$ with $a_{(i,j)} \in \lr$, we define:
\begin{rlist}
\item $E(a) = tr(a)$, 
\item $X_{n (i-1)+j} = e_{i,j}$,
\item $Y_{n (i-1)+j} = e_{j,i}$,
\end{rlist}
for $i,j= 1,2,\dots n$.
Then we have:
\begin{align*}
	\sum_{i=1}^{n}\sum_{j=1}^{n} E(aX_{n(i-1)+j})Y_{n(i-1)+j}&= \sum_{i=1}^{n}\sum_{j=1}^{n} E\left(\sum_{k=1}^{n}\sum_{l=1}^{n} a_{(k,l)}e_{k,l}e_{i,j}\right)Y_{n(i-1)+j}\\
	&=\sum_{i=1}^{n}\sum_{j=1}^{n} E\left(\sum_{k=1}^{n}  a_{(k,i)} e_{k,j}\right)Y_{n(i-1)+j}\\
	&=\sum_{i=1}^{n}\sum_{j=1}^{n}   a_{(j,i)} e_{j,i}\\
	&=a.
\end{align*}
Similarly, $\sum_{i=1}^{n}\sum_{j=1}^{n}X_{n(i-1)+j} E(Y_{n(i-1)+j}a)=a$.
Thus $M_n(\lr)/\lr$ is a  Frobenius extension.

Now, suppose that there exists an element $d\in R$ that satisfies following equations:
\begin{align*}
     \sum_{i=1}^{n}\sum_{j=1}^{n}X_{n (i-1)+j}\cdot d\cdot Y_{n (i-1)+j}&=I\\
     \sum_{i=1}^{n}\sum_{j=1}^{n}e_{i,j}\cdot d\cdot  e_{j,i}&=I\\
     d\cdot\sum_{i=1}^{n}\sum_{j=1}^{n}e_{i,j}\cdot   e_{j,i}&=I\\
     d\cdot\sum_{i=1}^{n}\sum_{j=1}^{n}e_{i,i}&=I\\
     d\cdot n\cdot \sum_{i=1}^{n}e_{i,i}&=I\\
     d\cdot n\cdot I&=I.
\end{align*}
Therefore, $M_n(\lr)/\lr$ is a separable Frobenius extension if and only if $dn=1$, that is, $n$ is invertible in $R$.
\endd
\vspace{10pt}

Proposition \ref{composite} states that the composition of two (separable) Frobenius extensions is also a (separable) Frobenius extension.
In general, the converse of Proposition \ref{composite} does not hold (see Example \ref{example1}).
However, we present a special case in which the converse does hold.

\begin{lemm}\label{inverse}
    Let $\cb/\lr$ be a Frobenius extension and  let $n$  be invertible in $R$.
    Then $M_n(\cb)/\lr$ is a separable Frobenius extension if and only if $\cb/\lr$ is a separable Frobenius extension.
\end{lemm}

\proof
We only claim that the Frobenius extension $\cb/\lr$ is  separable if the Frobenius extension $M_n(\cb)/\lr$ is separable.
Assume that $(E,X_k,Y_k)$ is a Frobenius system of $\cb/\lr$ with $1\leqslant k\leqslant m$.
For any matrix $b\in M_n(B)$ , we define:
\begin{center}
     $E'(b) =  E(tr(b))$,
\end{center}
\begin{center}
     $X'_{m(n(i-1)+j-1) + k} = X_ke_{i,j}$, $Y'_{m(n (i-1)+j-1)   + k} = Y_ke_{j,i}$.
\end{center}
Then $(E', X_i', Y_i')$ is a Frobenius system of $M_n(\cb)/\lr$, by Lemmas \ref{composite} and  \ref{RI,Mn}.
Since  $M_n(\cb)/\lr$ is a separable Frobenius extension, there exists $d = \sum_{p}^n\sum_q^n d_{(p,q)}e_{p,q}$ in $C_{M_n(\cb)}(\cb)$ with $d_{(p,q)}\in B$ such that:
$$I = \sum_i^n\sum_j^n\sum_k^m X'_{(n (i-1)+j-1) m + k}dY'_{(n (i-1)+j-1) m + k}.$$
Substituting the definitions of $X_i'$ and $Y_i'$, we have:
\begin{align*}
	I &= \sum_i^n\sum_j^n\sum_k^mX_ke_{i,j}\sum_p^n\sum_q^nd_{(p,q)}e_{p,q}Y_ke_{j,i}\\
	&= \sum_i^n\sum_j^n\sum_k^m X_ke_{i,j}d_{(j,j)}e_{j,j}Y_ke_{j,i}\\
	&= \sum_i^n\sum_k^m\sum_j^n X_k d_{(j,j)} Y_ke_{i,i} \\
	&= \sum_i^n\sum_k^m  X_k \sum_j^n d_{(j,j)} Y_k e_{i,i}.
\end{align*}
That is $\sum_k^m X_k\sum_j^n d_{(j,j)} Y_k =1$.
It is clear that $\sum_j^n d_{(j,j)}\in C_\cb(R)$.
Therefore $\cb/\lr$ is a separable Frobenius extension.
\endd
\vspace{10pt}

\no \emph{\textbf{Proof of Theorem \ref{main2}}}:

First we show that if   $n_i = n_j$ for $\lambda_i = \lambda_j$, then $S_n(c,\lr)/\lr$ is a Frobenius extension.

We denote by $\tela{J}_l(\tela{\lambda}_l)$ the direct sum of all Jordan blocks whose eigenvalues are $\tela{\lambda}_l$.
Based on the eigenvalues, we rewrite $\bigoplus J_{n_i}(\lambda_i)$ as:
$$\tela{J_1}(\tela{\lambda}_1)\bigoplus\tela{J_2}(\tela{\lambda}_2)\bigoplus\dots\bigoplus\tela{J_k}(\tela{\lambda}_k),$$ 
where $\tela{\lambda}_i\neq \tela{\lambda}_j$ if $i\neq j$.

Due to Lemma \ref{block}, we have:
$$S_n(\bigoplus J_{n_i}(\lambda_i),\lr) =   S_{n_1}(\tela{J_1}(\tela{\lambda}_1),\lr)\bigoplus S_{n_2}(\tela{J_2}(\tela{\lambda}_2),\lr)\bigoplus\dots\bigoplus S_{n_k} (\tela{J_k}(\tela{\lambda}_k),\lr).$$
Additionally, it follows from Lemma \ref{trivial} that $S_{n_i }(J_{n_i },\lr)/\lr$ is a Frobenius extension.
Assume that there are $m$ blocks whose eigenvalues are $\tela{\lambda}_l$.
By Lemma \ref{block}, $S_{m \tela{n}_l}(\tela{J}(\tela{\lambda}_l),\lr) = M_m(S_{\tela{n}_l}(J_{\tela{n}_l},\lr )).$
Due to Lemma \ref{composite}, the Frobenius extensions $M_m(S_{\tela{n}_l}(J_{\tela{n}_l},\lr ))/S_{\tela{n}_l}(J_{\tela{n}_l},\lr )$ and $S_{\tela{n}_l}(J_{\tela{n}_l},\lr )/\lr$ imply that   $M_m (S_{\tela{n}_l}(J_{\tela{n}_l}(\tela{\lambda}_l)))/\lr$ is a Frobenius extension.
That is $S_{m\tela{n}_l}(\tela{J}(\tela{\lambda}_l),\lr)/\lr$ is a Frobenius extension.
It follows from Lemma \ref{chai} that: 
$$(S_n(\tela{J_1}(\tela{\lambda}_1),\lr)\bigoplus S_n(\tela{J_2}(\tela{\lambda}_2),\lr)\bigoplus\dots\bigoplus S_n(\tela{J_k}(\tela{\lambda}_k),\lr))/\lr$$ 
is a Frobenius extension.

On the other hand, it is sufficient to prove that: 
$$S_{\sum_i n_i}(J_{n_1}(\lambda)\bigoplus\dots \bigoplus J_{n_k}(\lambda),\lr)/\lr$$ 
is a Frobenius extension  only if $n_1=\dots=n_k$.
Without loss of generality, suppose $n_1\leqslant n_2 \cdots \leqslant n_{k'}< n_{k'+1} = \dots =n_k$, i.e., there are $k-k'$   Jordan blocks of size $n_k\times n_k$.
For any $M\in  S_{\sum_i n_i}(\tela{J}(\lambda),\lr)$, where $\tela{J}(\lambda) = J_{n_1}(\lambda)\bigoplus\dots \bigoplus J_{n_k}(\lambda)$,
we write $$M=\begin{bmatrix}
    M_{1,1}&\dots & M_{1,k}\\
    \vdots& \ddots & \vdots\\
    M_{k,1} &\dots & M_{k,k}
\end{bmatrix}$$
with $M_{i,j}J_{n_j}=J_{n_i}M_{i,j}$.
By Lemma \ref{block2}, we have:
$$M_{i,j} = \sum_{l=\max\{0,n_j-n_i\}}^{n_j-1}  m_{(i,j)}^l \cdot  J_{n_i,n_j}^l .$$

Assume that $S_{\sum_i n_i}(\tela{J}(\lambda),\lr)/\lr$ is a Frobenius extension, there exists a Frobenius system $(E,\ _tX,\ _tY)$, where 
\begin{center}
    $E(M) = \sum_{i=1}^k\sum_{j=1}^k \sum_{l=\max\{0,n_j-n_i\}}^{n_j-1}  m_{(i,j)}^l \beta_{(i,j)}^l$,
\end{center}
\begin{center}
    $
\ _tX=\begin{bmatrix}
    \ _tX_{1,1}&\dots & \ _tX_{1,k}\\
    \vdots& \ddots & \vdots\\
    \ _tX_{k,1} &\dots & \ _tX_{k,k}
\end{bmatrix}$ with $\ _tX_{i,j} = \sum_{l=\max\{0,n_j-n_i\}}
^{n_j-1}  \ _tx_{(i,j)}^l  J_{n_i,n_j}^l$,
\end{center}
\begin{center}
    $\ _tY=\begin{bmatrix}
    \ _tY_{1,1}&\dots & \ _tY_{1,k}\\
    \vdots& \ddots & \vdots\\
    \ _tY_{k,1} &\dots & \ _tY_{k,k}
\end{bmatrix}$ with $\ _tY_{i,j} = \sum_{l=\max\{0,n_j-n_i\}}
^{n_j-1}  \ _ty_{(i,j)}^l  J_{n_i,n_j}^l$.
\end{center}
By Lemma \ref{mul}, we have
\begin{center}
    $J_{n_i,n_k}^{n_k-1}\ _tX_{k,j}=\left\{\begin{aligned}
    0\quad\quad& &j\leqslant k'\\
    x_{(k,j)}^0J_{n_i,n_j}^{n_k-1}& &j> k'
\end{aligned} \right. $. \quad\quad$(*)$
\end{center}

Then, set $$A= \begin{bmatrix}
    0&\dots &0&   \alpha^1  J_{n_1,n_k}^{n_k-1}\\
    \vdots& \ddots &\vdots& \vdots\\
    0 &\dots &0 &   \alpha^k  J_{n_k,n_k}^{n_k-1}
\end{bmatrix}\in S_{\sum_i n_i}(\tela{J},\lr).$$
By assumption, we have:
\begin{align*}
    A&=\sum_t E(A\ _tX)\ _tY \\
    &=\sum_t E\left(\begin{bmatrix}
        0&\dots &0& \alpha^1  J_{n_1,n_k}^{n_k-1}\\
        \vdots& \ddots &\vdots& \vdots\\
        0 &\dots &0 & \alpha^k  J_{n_k,n_k}^{n_k-1}
    \end{bmatrix}\begin{bmatrix}
        \ _tX_{1,1}&\dots & \ _tX_{1,k}\\
        \vdots& \ddots & \vdots\\
        \ _tX_{k,1} &\dots & \ _tX_{k,k}
    \end{bmatrix}\right)\ _tY \\
    &=\sum_t E\left(\begin{bmatrix}
        \alpha^1  J_{n_1,n_k}^{n_k-1}\ _tX_{k,1}&\dots & \alpha^1  J_{n_1,n_k}^{n_k-1}\ _tX_{k,k}\\
        \vdots& \ddots & \vdots\\
        \alpha^k  J_{n_k,n_k}^{n_k-1}\ _tX_{k,1} &\dots  & \alpha^k  J_{n_k,n_k}^{n_k-1}\ _tX_{k,k}
    \end{bmatrix}\right)\ _tY \\
    &\overset{\text{by } (*)}{=}\sum_t E\left(\begin{bmatrix}
        0&\dots &0&\alpha^1\  _tx_{(k,k'+1)}^0 J_{n_1,n_{k'+1}}^{n_k-1} &\dots & \alpha^1 \  _tx_{(k,k)}^0  J_{n_k,n_k}^{n_k-1} \\
        \vdots&\ddots&\vdots&\vdots& \ddots & \vdots\\
        0&\dots &0&\alpha^k\  _tx_{(k,k'+1)}^0  J_{n_1,n_{k'+1}}^{n_k-1}  &\dots  & \alpha^k \  _tx_{(k,k)}^0  J_{n_k,n_k}^{n_k-1}
    \end{bmatrix}\right)\ _tY\\
    &= \sum_t \left(\sum_{p=1}^{k}\sum_{q=k'+1}^{k}   \alpha^p\  _tx_{(k,q)}^0  \beta_{(p,q)}^{n_k-1}\right)  \ _tY\\
    &= \sum_{p=1}^{k} \alpha^p\sum_{q=k'+1}^{k}   \beta_{(p,q)}^{n_k-1}\sum_t\  _tx_{(k,q)}^0   \begin{bmatrix}
        \ _tY_{1,1}&\dots & \ _tY_{1,k}\\
        \vdots& \ddots & \vdots\\
        \ _tY_{k,1} &\dots & \ _tY_{k,k}
    \end{bmatrix}.
\end{align*}
That is,
$$\sum_{p=1}^{k} \alpha^p\sum_{q=k'+1}^{k}   \beta_{(p,q)}^{n_k-1}\sum_t\  _tx_{(k,q)}^0 \ _tY_{i,k}=\alpha^i  J_{n_i,n_k}^{n_k-1}.$$
So, we have the following equations:
\begin{align*}
    \sum_{p=1}^{k}  \alpha^p\sum_{q=k'+1}^{k}   \beta_{(p,q)}^{n_k-1}\sum_t\  _tx_{(k,q)}^0\cdot\ _t y_{(1,k)}^{n_k-1} &=   \alpha^1+    \alpha^2\cdotp0+ \dots    +\alpha^k\cdotp0,\\
     &  \vdots\\
    \sum_{p=1}^{k} \alpha^p\sum_{q=k'+1}^{k}  \beta_{(p,q)}^{n_k-1}\sum_t\  _tx_{(k,q)}^0\cdot\ _t y_{(k,k)}^{n_k-1} &=  \alpha^1\cdotp0+ \dots  +\alpha^{k-1}\cdotp0+   \alpha^k.
\end{align*}
Rewrite the above equations in matrix form:
$$
\begin{bmatrix}
    \alpha^1  \beta_{(1,k'+1)}^{n_k-1}& \dots &   \alpha^1  \beta_{(1,k)}^{n_k-1}\\
    \vdots& \ddots& \vdots\\
    \  \alpha^k  \beta_{(k,k'+1)}^{n_k-1}&\dots &  \alpha^k  \beta_{(k,k)}^{n_k-1}
\end{bmatrix}
\begin{bmatrix}
    \sum_t\  _tx_{(k,k'+1)}^0\cdot\ _t y_{(1,k)}^{n_k-1}&\dots & \sum_t\  _tx_{(k,k'+1)}^0\cdot\ _t y_{(k,k)}^{n_k-1}\\
    \vdots& \ddots & \vdots\\
    \sum_t\  _tx_{(k,k)}^0\cdot\ _t y_{(1,k)}^{n_k-1} &\dots & \sum_t\  _tx_{(k,k)}^0\cdot\ _t y_{(k,k)}^{n_k-1}
\end{bmatrix}$$
\begin{flushright}
    $=\begin{bmatrix}
        \   \alpha^1& &0\\
        &\ddots&\\
        0& &   \alpha^k
    \end{bmatrix}$.
\end{flushright}

Note that the scalars $\alpha^1\dots \alpha^k$ are arbitrary.
Thus, there is a contradiction here when comparing the ranks of the two sides of the above equation.
Therefore,  the assumption is incorrect. 
That is,
$S_{\sum_i n_i}(J_{n_1}(\lambda)\bigoplus\dots \bigoplus   J_{n_k}(\lambda) ,\lr)/\lr$ is a Frobenius extension  only if $n_1=\dots=n_k$.

Now, we turn to the proof of  (2).

If $n_i=1$ for any $i$, then $\tela{J}_l(\tela{\lambda}_l) = \tela{\lambda}_l I$, where $I$ is the identity matrix. 
This implies that $ S_m(\tela{J_l},\lr) = M_m(\lr)$. 
By Lemma \ref{RI,Mn}, it follows that $ S_m(\tela{J},\lr)/\lr$ is a separable Frobenius extension.
On the other hand, if there exists some $\tela{n}_l>1$,  then $S_{\tela{n}_l}(J_{\tela{n}_l},\lr)/\lr$ is not a separable Frobenius extension.
Consequently, by Lemma \ref{inverse}, this implies that $M_m(S_{\tela{n}_l}(J_{\tela{n}_l},\lr))/\lr$ is not a separable Frobenius extension.
Therefore, by Lemma \ref{chai}, it follows that $S_n(J,R)/R$ is not a separable Frobenius extension.
\endd
\vspace{10pt}

\section{The proof of Corollary \ref{coro0}}\label{2sec}

Xi and Zhang claim in \cite{MR4241258} that for a matrix $a\in M_n(\lr)$, which is similar to a matrix $c\in M_n(\lr)$,  the extension $M_n(\lr)/S_n(c, \lr)$ is a separable Frobenius extension if and only if $ M_n(\lr)/S_n(a, \lr)$ is also a separable Frobenius extension.
There are also   similar results for the extension $S_n(c,\lr)/\lr$, as described below.

\begin{lemm}\label{similar}
Let $\lr$ be a commutative algebra, $c,a\in M_n(R)$ and $u\in GL_n(R)$ with $a = u^{-1}cu$.
Then $S_n(c,\lr)/\lr$ is a (separable) Frobenius extension if and only if $S_n(a,\lr)/ \lr$ is a (separable) Frobenius extension.
\end{lemm}

\proof
It is sufficient to show that if $S_n(c,\lr)/\lr$ is a (separable) Frobenius extension, then $S_n(a,\lr)/\lr$ is also a (separable) Frobenius extension.

First, we assume that  $S_n(c,\lr)/\lr$ is a  Frobenius extension.
Then, there exists a Frobenius system $(E^c,X^c_i,Y^c_i)$.
Define
$E^a: S_n(a,\lr)\ra \lr$ via $E^a(b)  = u^{-1}E^c(ubu^{-1})u$ and $X^a_i = u^{-1}X^c_iu$, $Y^a_i = u^{-1}Y^c_iu.$
Then we claim that  $(E^a,X^a_i,Y^a_i)$ is a Frobenius system.

For any $\gamma_1,\gamma_2\in\lr$ and $b, b'\in  S_n(a,\lr)$,
\begin{align*}
	E^a(\gamma_1 b\gamma_2 ) &= u^{-1}E^c(u\gamma_1 b\gamma_2 u^{-1})u\\
	&=u^{-1}E^c(\gamma_1 ubu^{-1}\gamma_2 )u\\
	&=u^{-1}\gamma_1 E^c(ubu^{-1})\gamma_2 u\\
	&=\gamma_1 u^{-1}E^c(ubu^{-1})u\gamma_2 \\
	&=\gamma_1 E^a(b)\gamma_2;
\end{align*}
\begin{align*}
	E^a(b +b') &= u^{-1}E^c(u( b+b')u^{-1})u\\
	&=u^{-1}E^c(u(b)u^{-1}+u(b')u^{-1})u\\
	&=u^{-1}E^c(u(b)u^{-1})u+u^{-1}E^c(u(b')u^{-1})u\\
	&=E^a(b)+E^a(b');
\end{align*}
and
\begin{align*}
	\sum_i E^a(bX^a_i)Y^a_i &= \sum_i u^{-1}E^c(ubX^a_iu^{-1})u\cdot u^{-1}Y^c_iu\\
	&=u^{-1}\left(\sum_i E^c(ubu^{-1}X^c_i)Y^c_i \right)u\\
	&=u^{-1}ubu^{-1}u\\
	&=b.
\end{align*}
Similarly, we have
$\sum_i X^a_iE^a(Y^a_ib) =b$.
Thus, $S_n(a,\lr)/\lr$ is a  Frobenius extension.

Finally, we assume that there exists $d\in S_n(c,\lr)$, such that $\sum X^c_i d Y^c_i =I$.
Then, 
\begin{center}
	$\sum_i X^a_i u^{-1}d u Y^a_i = \sum_i u^{-1} X^c_i u\cdot u^{-1} d u\cdot u^{-1} Y^c_i u = I.$
\end{center}
Therefore, $S_n(a,\lr)/\lr$ is a  separable  Frobenius extension.
\endd
\vspace{10pt}

The following results are well-known.

\begin{lemm}\label{0-tensor}
    Let $A$ and $B$ be algebras over a field $R$.
    For any $a\!\otimes_R\! b\in A\!\otimes_R\! B$, $a\!\otimes_R\! b =0 $ if and only if $a=0$ or $b=0$.
\end{lemm}

For a matrix $c$, it does not necessarily have a Jordan canonical form because its eigenvalues may lie in the algebraic closure $\overline{R}$ of $R$ and not necessarily in $R$ itself.
Therefore, the following lemma discusses the impact of field extensions on separable Frobenius extensions.

\begin{lemm}\label{tensor}
    Let $A$ be an algebra over a field $R$ and
    let $K$ be a field extension of $R$.
    Then the extension $(A\otimes_R K)/( R\otimes_R K)$ is a separable Frobenius extension if and only if the extension $A/R$ is a separable Frobenius extension.
\end{lemm}
\proof
It  suffices to show that the extension  $A/R$ is a separable Frobenius extension if $(A\otimes_R K)/( R\otimes_R K)$ is a separable Frobenius extension by \cite[Lemma 3.3(1)]{MR4461655}.
Let $\{a_i\}$ be a basis of $A$ and $\{k_j\}$ be a basis of $K$ over $R$.
According to \cite[Proposition 3.1]{MR4461655}, the extension $A/R$ is a  Frobenius extension if and only if the extension $(A\otimes_k K)/(R\otimes_k K)$ is a  Frobenius extension.
Assume that $(E,X_l,Y_l)$ is a Frobenius system of the extension $A/R$.

Define $\overline{E}(a_i\otimes_R k_j) = E(a_i)\otimes_R k_j$,  $\overline{X_i}=X_i\otimes_R 1$, and $\overline{Y_i}=Y_i\otimes_R 1$.

Next we claim that $(\overline{E},\overline{X_i},\overline{Y_i})$ is a Frobenius system of $(A\otimes_R K)/( R\otimes_R K)$.
First, $\overline{E}$ is well defined and is an $R\otimes_R K$-$R\otimes_R K$-bimodule morphism.
For any $\alpha_i a_i\otimes_R \beta_j k_j=\alpha_i' a_i\otimes_R \beta_j' k_j$ with $\alpha_i,\alpha_i',\beta_j,\beta_j'\in R$,  it follows that  $\alpha_i\beta_j = \alpha_i'\beta_j' $.
Then we have:
\begin{align*}
    \overline{E}[(\alpha_i a_i)\otimes_R (\beta_j k_j)]  &= E(\alpha_i a_i)\otimes_R (\beta_j k_j)\\
    & =[E( a_i)\alpha_i]\otimes_R (\beta_j k_j)\\
     &=E( a_i)\otimes_R (\alpha_i\beta_j k_j)\\
         &=E( a_i)\otimes_R (\alpha_i'\beta_j' k_j)\\
    &=[E( a_i)\alpha_i']\otimes_R (\beta_j' k_j)\\
    &=E( \alpha_i'a_i) \otimes_R (\beta_j' k_j)\\
    &=\overline{E}[( \alpha_i'a_i) \otimes_R (\beta_j' k_j)].
\end{align*}
Further, for any $\sum_m \gamma_m\otimes_R (\beta_m k_m),  \sum_n \gamma_n\otimes_R (\beta_n k_n)\in  R\otimes_R K$, $\sum_i\sum_j(\alpha_{(i,j)} a_i)\otimes_R (\beta_{(i,j)} k_j)\in  A\otimes_R K$ with $\gamma_m, \gamma_n, \alpha_{(i,j)}, \beta_m, \beta_n, \beta_{(i,j)}\in R$, we have:
\begin{align*}
    &\overline{E}\left( \left( \sum_m \gamma_m\otimes_R (\beta_m k_m)\right)\left( \sum_{i,j}(\alpha_{(i,j)} a_i)\otimes_R (\beta_{(i,j)} k_j)\right)\left( \sum_n \gamma_n\otimes_R (\beta_n k_n)\right)\right) \\
    =&  \overline{E}\left( \sum_{m,i,j,n}(\gamma_m\alpha_{(i,j)}  a_i\gamma_n)\otimes_R (\beta_mk_m\beta_{(i,j)} k_j \beta_n  k_n)\right) \\
    =& \sum_{m,i,j,n}\overline{E}( (\gamma_m\alpha_{(i,j)}  a_i\gamma_n)\otimes_R (\beta_mk_m\beta_{(i,j)} k_j \beta_n  k_n))\\
    = &\sum_{m,i,j,n}E(\gamma_m\alpha_{(i,j)}  a_i\gamma_n)\otimes_R (\beta_mk_m\beta_{(i,j)} k_j \beta_n  k_n)\\
    = &\sum_{m,i,j,n}\gamma_mE(\alpha_{(i,j)}  a_i)\gamma_n\otimes_R (\beta_mk_m\beta_{(i,j)} k_j \beta_n  k_n)\\
    = &\left( \sum_m\gamma_m\otimes_R (\beta_mk_m)\right)\left( \sum_i\sum_jE(\alpha_{(i,j)}  a_i)\otimes_R(\beta_{(i,j)} k_j) \right)\left( \sum_n\gamma_n\otimes_R (\beta_n  k_n)\right)\\
    = &\left( \sum_m\gamma_m\otimes_R (\beta_mk_m)\right)\left( \sum_{i,j}\overline{E}( (\alpha_{(i,j)}  a_i)\otimes_R(\beta_{(i,j)} k_j)) \right)\left( \sum_n\gamma_n\otimes_R (\beta_n  k_n)\right)\\
    = &\left( \sum_m\gamma_m\otimes_R (\beta_mk_m)\right)\left( \overline{E}\left( \sum_{i,j}(\alpha_{(i,j)}  a_i)\otimes_R(\beta_{(i,j)} k_j)\right) \right)\left( \sum_n\gamma_n\otimes_R (\beta_n  k_n)\right).
\end{align*}

Moreover, for any $ \sum_{i,j}(\alpha_{(i,j)}  a_i)\otimes_R(\beta_{(i,j)} k_j) \in A\otimes_R K$, we have
\begin{align*}
    &\sum_l \overline{E}\left(\left(\sum_{i,j}(\alpha_{(i,j)}  a_i)\otimes_R(\beta_{(i,j)} k_j)\right) (X_l\otimes_R 1)\right)(Y_l\otimes_R 1)\\
    =&\sum_l \overline{E}\left(\sum_{i,j}(\alpha_{(i,j)}  a_iX_l)\otimes_R(\beta_{(i,j)} k_j)\right)(Y_l\otimes_R 1) \\
    =&\sum_l \left(\sum_{i,j}\overline{E}((\alpha_{(i,j)}  a_iX_l)\otimes_R(\beta_{(i,j)} k_j))\right)(Y_l\otimes_R 1) \\
    =&\sum_{i,j,l}\overline{E}((\alpha_{(i,j)}  a_iX_l)\otimes_R(\beta_{(i,j)} k_j))(Y_l\otimes_R 1) \\
    =& \sum_{i,j,l}(E(\alpha_{(i,j)}  a_iX_l)\otimes_R(\beta_{(i,j)} k_j)) (Y_l\otimes_R 1) \\
    =& \sum_{i,j,l} (E(\alpha_{(i,j)}  a_iX_l)Y_l)\otimes_R(\beta_{(i,j)} k_j)  \\
    =& \sum_{i,j} \left(\sum_lE(\alpha_{(i,j)}  a_iX_l)Y_l\right)\otimes_R(\beta_{(i,j)} k_j)  \\
    =& \sum_{i,j}   (\alpha_{(i,j)}  a_i) \otimes_R(\beta_{(i,j)} k_j).
\end{align*}
Similarly, we have:
\begin{align*}
    &\sum_l (X_l\otimes_R 1)\overline{E}\left((Y_l\otimes_R 1)\left(\sum_{i,j}(\alpha_{(i,j)}  a_i)\otimes_R(\beta_{(i,j)} k_j)\right) \right)\\
    =&\sum_l (X_l\otimes_R 1)\overline{E}\left(\sum_{i,j}(Y_l\alpha_{(i,j)}  a_i)\otimes_R(\beta_{(i,j)} k_j)\right) \\
    =&\sum_l (X_l\otimes_R 1)\left(\sum_{i,j}\overline{E}((Y_l\alpha_{(i,j)}  a_i)\otimes_R(\beta_{(i,j)} k_j))\right) \\
    =&\sum_{i,j,l}(X_l\otimes_R 1)\overline{E}((Y_l\alpha_{(i,j)}  a_i)\otimes_R(\beta_{(i,j)} k_j)) \\
    =& \sum_{i,j,l}(X_l\otimes_R 1)(E(Y_l\alpha_{(i,j)}  a_i)\otimes_R(\beta_{(i,j)} k_j))  \\
    =& \sum_{i,j,l} (X_lE(Y_l\alpha_{(i,j)}  a_i))\otimes_R(\beta_{(i,j)} k_j)  \\
    =& \sum_{i,j}\left( \sum_lX_lE(Y_l\alpha_{(i,j)}  a_i)\right)\otimes_R(\beta_{(i,j)} k_j)  \\
    =& \sum_{i,j}   (\alpha_{(i,j)}  a_i) \otimes_R(\beta_{(i,j)} k_j).
\end{align*}

Since $(A\otimes_R K)/( R\otimes_R K)$ is a separable Frobenius extension, there exists an element\\ $\sum_{i,j}(\lambda_{(i,j)}a_i)\otimes_R (\gamma_{(i,j)} k_j) \in A\otimes_R K$ such that:
\begin{center}
    $\sum_l (X_l\otimes_R 1)(\sum_{i,j}(\lambda_{(i,j)}a_i)\otimes_R (\gamma_{(i,j)} k_j)](Y_l\otimes_R 1)=1\otimes_R 1$.
\end{center}
Then we have:
\begin{align*}
    1\otimes_R 1&=\sum_l (X_l\otimes_R 1)\left(\sum_{i,j}\lambda_{(i,j)}a_i\otimes_R \gamma_{(i,j)} k_j\right)(Y_l\otimes_R 1)\\
    &=\sum_l \left(\sum_{i,j}(X_l \lambda_{(i,j)}a_iY_l)\otimes_R (\gamma_{(i,j)} k_j)\right)  \\
    &=\sum_{i,j,l}(X_l \lambda_{(i,j)}a_iY_l)\otimes_R (\gamma_{(i,j)} k_j)\\
    &=\sum_{i,j}\left( \sum_l (X_l \lambda_{(i,j)}a_iY_l)\right)\otimes_R (\gamma_{(i,j)} k_j).
\end{align*}

Assume that $1_K = \sum_j \overline{\gamma_j}k_j$.
Since $1_K\neq 0$, there exists some $\overline{\gamma_m}\neq 0$.
We then consider the terms associated with \(k_m\), leading to the following equation:
\begin{align*}
    1\otimes_R \overline{\gamma_m}k_m  &=\sum_i\left( \sum_l (X_l \lambda_{(i,m)}a_iY_l)\right)\otimes_R (\gamma_{(i,m)} k_m)\\
    1\otimes_R \overline{\gamma_m}k_m  &=\sum_{i,l} X_l \lambda_{(i,m)}a_iY_l\gamma_{(i,m)}(\overline{\gamma_m})^{-1}\otimes_R \overline{\gamma_m}k_m\\
    0&=\left(\sum_{i,l} X_l \lambda_{(i,m)}\gamma_{(i,m)}(\overline{\gamma_m})^{-1}a_iY_l-1\right)\otimes_R \overline{\gamma_m}k_m. 
\end{align*}

By applying Lemma \ref{0-tensor},
we obtain the following equations:
\begin{align*}
    \sum_{i,l}  X_l \lambda_{(i,m)}\gamma_{(i,m)}(\overline{\gamma_m})^{-1}a_iY_l  &= 1_A\\
    \sum_l  X_l \left(\sum_i\lambda_{(i,m)}\gamma_{(i,m)}(\overline{\gamma_m})^{-1}a_i\right)Y_l &= 1_A.
\end{align*}
That is, there exists an element $a = \sum_i\lambda_{(i,m)}\gamma_{(i,m)}(\overline{\gamma_m})^{-1}a_i \in A$ such that $\sum_l  X_l aY_l = 1_A$.
Therefore, $A/R$ is a separable Frobenius extension.
\endd
\vspace{10pt}

\no \emph{\textbf{Proof of Corollary \ref{coro0}}}:

Suppose that $R$ is a field and $c\in M_n(R)$.
Let $\overline{R}$ denote an algebraic closure of $R$.
We consider the smallest subfield $K\subset\overline{R}$ containing $R$ and all eigenvalues of $c$; that is, $K$ is obtained by adjoining to $R$ all eigenvalues of $c$.
Consequently, $K$ is a finite extension of $R$, i.e., $\dim_R(K) < \infty$.
We may assume that $K$ is the splitting field of the characteristic polynomial of $c$, so that $c$ is similar to a Jordan-block matrix via a matrix in $GL_n(K)$.
Since $K/R$ is a field extension, we have $R \otimes_R K \cong K$.
By \cite[Proposition 3.1]{MR4461655},  \cite[Lemma 3.4]{MR4461655} and Lemma \ref{tensor}, we deduce that the extension $S_n(c,R)/R$ is (separable) Frobenius if and only if so is $S_n(c,K)/K$.
Furthermore, by Theorem \ref{main2} and Lemma \ref{similar}, we conclude that $S_n(c,K)/K$ is a Frobenius extension  if $n_i = n_j$ whenever $\lambda_i = \lambda_j$.
Moreover, $S_n(c,K)/K$ is a  separable  Frobenius extension if and only if $n_i = 1$ for every $i$.
\endd
\vspace{10pt}

Here we provide some examples of the Frobenius extension using Corollary \ref{coro0}.

\begin{exam}\label{example1}
	Let $\mr$ be the field of real numbers.
	Then we have the following:
	\begin{rlist}
		\item Let $a = \begin{bmatrix}
			 0&	0&	1\\
			0&	1&	0\\
			-1&	0&	2
		\end{bmatrix} \in M_3(\mr).$
		Then
		$S_n(c,\mr)/\mr$ is not a Frobenius extension, because the Jordan standard form of $a$ is  $\begin{bmatrix}
			1 & 1 &0\\
			0& 1&0\\
			0&0&1
		\end{bmatrix}.$
		\item Let $b = \begin{bmatrix}
			0&	0&	1\\
			0&	0&	0\\
			-1&	0&	2
		\end{bmatrix} \in M_3(\mr).$
		Then
		$S_n(b,\mr)/\mr$ is a Frobenius extension, because the Jordan standard form of $b$ is  $\begin{bmatrix}
			1 & 1 &0\\
			0& 1&0\\
			0&0&0
		\end{bmatrix}.$
		\item Let $c = \begin{bmatrix}
		0 & 1\\
		-1& 0
		\end{bmatrix} \in M_2(\mr).$
		Then
		$S_n(c,\mr)/\mr$ is a separable Frobenius extension, because the Jordan standard form of $c$ is  $\begin{bmatrix}
		i & 0\\
		0& -i
		\end{bmatrix}.$
	\end{rlist}

\end{exam}

Moreover, we obtain the following corollary.

\begin{coro}
	Let $R$ be a field of characteristic not equal to 2.
	For any matrix $c\in M_2(R)$, the extension $S_2(c,R)/\lr$ is a Frobenius extension.
\end{coro}
\proof
For any $2 \times 2$ matrix $c$, the Jordan blocks in its Jordan canonical form (over an algebraic closure of $R$) cannot differ in size.
Therefore, by Corollary \ref{coro0}, the extension $S_2(c,R)/\lr$ is Frobenius.

\endd

\section*{Acknowledgments}

The authors would like to thank Changchang Xi for interesting discussions.
The authors are also grateful to the referees for the valuable comments and suggestions, which helped to improve the presentation of the paper.

\newpage
\bibliography{tex}

@misc{xizhang2022,
      title={New invariants of stable equivalences of algebras}, 
      author={Changchang Xi and Jinbi Zhang},
      year={2023},
      eprint={2207.10848},
      archivePrefix={arXiv},
      primaryClass={math.RT},
}

@book {1986,
	AUTHOR = {Gohberg, Israel and Lancaster, Peter and Rodman, Leiba},
	TITLE = {Invariant subspaces of matrices with applications},
	SERIES = {Classics in Applied Mathematics},
	VOLUME = {51},
	NOTE = {Reprint of the 1986 original},
	PUBLISHER = {Society for Industrial and Applied Mathematics (SIAM),
	Philadelphia, PA},
	YEAR = {2006},
	PAGES = {15--1},
	ISBN = {0-89871-608-X},
	MRCLASS = {15-01 (15-02 47A15)},
	MRNUMBER = {2228089},
	MRREVIEWER = {Edward Azoff},
}

@article {MR4461655,
	AUTHOR = {Xi, Changchang and Zhang, Jinbi},
	TITLE = {Centralizer matrix algebras and symmetric polynomials of
	partitions},
	JOURNAL = {J. Algebra},
	FJOURNAL = {Journal of Algebra},
	VOLUME = {609},
	YEAR = {2022},
	PAGES = {688--717},
	ISSN = {0021-8693},
	MRCLASS = {16W22 (05A17 15A30 16K99 16S50)},
	MRNUMBER = {4461655},
	MRREVIEWER = {Mykola Khrypchenko},
}

@article {MR3018047,
	AUTHOR = {Dolinar, Gregor and Guterman, Alexander and Kuzma, Bojan and
	Oblak, Polona},
	TITLE = {Extremal matrix centralizers},
	JOURNAL = {Linear Algebra Appl.},
	FJOURNAL = {Linear Algebra and its Applications},
	VOLUME = {438},
	YEAR = {2013},
	NUMBER = {7},
	PAGES = {2904--2910},
	ISSN = {0024-3795},
	MRCLASS = {15A27 (15A18 15A30)},
	MRNUMBER = {3018047},
	MRREVIEWER = {Olga Victorovna Markova},
}

@article {MR1401512,
	AUTHOR = {Beidar, K. I. and Fong, Y. and Stolin, A.},
	TITLE = {On {F}robenius algebras and the quantum {Y}ang-{B}axter
	equation},
	JOURNAL = {Trans. Amer. Math. Soc.},
	FJOURNAL = {Transactions of the American Mathematical Society},
	VOLUME = {349},
	YEAR = {1997},
	NUMBER = {9},
	PAGES = {3823--3836},
	ISSN = {0002-9947},
	MRCLASS = {16L60 (81R50)},
	MRNUMBER = {1401512},
	MRREVIEWER = {Zhong Qi Ma},
}

@article {MR998028,
	AUTHOR = {Datta, Lokesh and Morgera, Salvatore D.},
	TITLE = {On the reducibility of centrosymmetric matrices---applications
	in engineering problems},
	JOURNAL = {Circuits Systems Signal Process.},
	FJOURNAL = {Circuits, Systems, and Signal Processing},
	VOLUME = {8},
	YEAR = {1989},
	NUMBER = {1},
	PAGES = {71--96},
	ISSN = {0278-081X},
	MRCLASS = {15A21 (15A09 15A15 15A57 15A90)},
	MRNUMBER = {998028},
	MRREVIEWER = {Frank Uhlig},
}

@article {MR820054,
	AUTHOR = {Weaver, James R.},
	TITLE = {Centrosymmetric (cross-symmetric) matrices, their basic
	properties, eigenvalues, and eigenvectors},
	JOURNAL = {Amer. Math. Monthly},
	FJOURNAL = {American Mathematical Monthly},
	VOLUME = {92},
	YEAR = {1985},
	NUMBER = {10},
	PAGES = {711--717},
	ISSN = {0002-9890},
	MRCLASS = {15A57 (15A18 15A21)},
	MRNUMBER = {820054},
	MRREVIEWER = {John Chollet},
}

@article {MR2018787,
	AUTHOR = {Premet, Alexander},
	TITLE = {Nilpotent commuting varieties of reductive {L}ie algebras},
	JOURNAL = {Invent. Math.},
	FJOURNAL = {Inventiones Mathematicae},
	VOLUME = {154},
	YEAR = {2003},
	NUMBER = {3},
	PAGES = {653--683},
	ISSN = {0020-9910},
	MRCLASS = {20G05 (17B50)},
	MRNUMBER = {2018787},
	MRREVIEWER = {George J. McNinch},
}

@article {MR2363491,
	AUTHOR = {Panyushev, Dmitri I.},
	TITLE = {Two results on centralisers of nilpotent elements},
	JOURNAL = {J. Pure Appl. Algebra},
	FJOURNAL = {Journal of Pure and Applied Algebra},
	VOLUME = {212},
	YEAR = {2008},
	NUMBER = {4},
	PAGES = {774--779},
	ISSN = {0022-4049},
	MRCLASS = {17B20 (14L30 20G20)},
	MRNUMBER = {2363491},
	MRREVIEWER = {Simon M. Goodwin},
}

@book {MR201472,
	AUTHOR = {Suprunenko, D. A. and Ty\v{s}kevi\v{c}, R. I.},
	TITLE = {Commutative Matrices},
	PUBLISHER = {``Nauka i Tekhnika'', Minsk},
	YEAR = {1966},
	PAGES = {104},
	MRCLASS = {16.48},
	MRNUMBER = {201472},
	MRREVIEWER = {R. F. Rinehart},
}

@article {MR85214,
	AUTHOR = {Taussky, Olga},
	TITLE = {Commutativity in finite matrices},
	JOURNAL = {Amer. Math. Monthly},
	FJOURNAL = {American Mathematical Monthly},
	VOLUME = {64},
	YEAR = {1957},
	PAGES = {229--235},
	ISSN = {0002-9890},
	MRCLASS = {15.0X},
	MRNUMBER = {85214},
	MRREVIEWER = {A. Brauer},
}

@book {MR168568,
	AUTHOR = {Wedderburn, J. H. M.},
	TITLE = {Lectures on matrices},
	PUBLISHER = {Dover Publications, Inc., New York},
	YEAR = {1964},
	PAGES = {vii+200pp},
	MRCLASS = {15.00 (01.60)},
	MRNUMBER = {168568},
}

@article {xi1,
	AUTHOR = {Xi, Changchang and Yin, Shujun},
	TITLE = {Cellularity of centrosymmetric matrix algebras and {F}robenius
		extensions},
	JOURNAL = {Linear Algebra Appl.},
	FJOURNAL = {Linear Algebra and its Applications},
	VOLUME = {590},
	YEAR = {2020},
	PAGES = {317--329},
	ISSN = {0024-3795},
	MRCLASS = {16S50 (15B57 16W10 19D50)},
	MRNUMBER = {4053642},
	MRREVIEWER = {Thiago Castilho de Mello},
}

@book {newexam,
	AUTHOR = {Kadison, Lars},
	TITLE = {New examples of {F}robenius extensions},
	SERIES = {University Lecture Series},
	VOLUME = {14},
	PUBLISHER = {American Mathematical Society, Providence, RI},
	YEAR = {1999},
	PAGES = {x+84},
	ISBN = {0-8218-1962-3},
	MRCLASS = {16L60 (16H05 16W30 46L37 57M25 81T45)},
	MRNUMBER = {1690111},
	MRREVIEWER = {J\"{o}rg Feldvoss},
}

@article {MR4241258,
	AUTHOR = {Xi, Changchang and Zhang, Jinbi},
	TITLE = {Structure of centralizer matrix algebras},
	JOURNAL = {Linear Algebra Appl.},
	FJOURNAL = {Linear Algebra and its Applications},
	VOLUME = {622},
	YEAR = {2021},
	PAGES = {215--249},
	ISSN = {0024-3795},
	MRCLASS = {16S50 (15A27 15B33 16U70 16W22 20C05)},
	MRNUMBER = {4241258},
	MRREVIEWER = {A. Joseph Kennedy},
}

@article {Xi-fbfd,
	AUTHOR = {Xi, Changchang},
	TITLE = {Frobenius bimodules and flat-dominant dimensions},
	JOURNAL = {Sci. China Math.},
	FJOURNAL = {Science China. Mathematics},
	VOLUME = {64},
	YEAR = {2021},
	NUMBER = {1},
	PAGES = {33--44},
	ISSN = {1674-7283},
	MRCLASS = {16D20 (16E10 16S50 17B35 17B37 18G20)},
	MRNUMBER = {4200983},
	MRREVIEWER = {Alex S. Dugas},
}

@book {Kock-fatqft,
	AUTHOR = {Kock, Joachim},
	TITLE = {Frobenius algebras and 2{D} topological quantum field
	theories},
	SERIES = {London Mathematical Society Student Texts},
	VOLUME = {59},
	PUBLISHER = {Cambridge University Press, Cambridge},
	YEAR = {2004},
	PAGES = {xiv+240},
	ISBN = {0-521-83267-5; 0-521-54031-3},
	MRCLASS = {57R56 (16L60 18D10 53D45 81T45)},
	MRNUMBER = {2037238},
	MRREVIEWER = {Lowell Abrams},
}

@incollection {Gnilke-etbwfrfb,
	AUTHOR = {Gnilke, Oliver W. and Greferath, Marcus and Honold, Thomas and
	Wood, Jay A. and Zumbr\"{a}gel, Jens},
	TITLE = {The extension theorem for bi-invariant weights over
	{F}robenius rings and {F}robenius bimodules},
	BOOKTITLE = {Rings, modules and codes},
	SERIES = {Contemp. Math.},
	VOLUME = {727},
	PAGES = {117--129},
	PUBLISHER = {Amer. Math. Soc., [Providence], RI},
	YEAR = {[2019] \copyright 2019},
	MRCLASS = {94B05 (16P10)},
	MRNUMBER = {3938144},
	MRREVIEWER = {Edgar Mart\'{\i}nez-Moro},
}

@article {Kanzaki-ngffe,
	AUTHOR = {Kanzaki, Teruo},
	TITLE = {A note on {W}itt groups for a {F}robenius extension with
	involution},
	JOURNAL = {J. Pure Appl. Algebra},
	FJOURNAL = {Journal of Pure and Applied Algebra},
	VOLUME = {22},
	YEAR = {1981},
	NUMBER = {3},
	PAGES = {249--252},
	ISSN = {0022-4049},
	MRCLASS = {15A63 (16A36 16A42)},
	MRNUMBER = {629333},
	MRREVIEWER = {Alex Rosenberg},
}

\vspace{4mm}

\end{document}